\documentclass[12pt,a4paper]{article}
\usepackage[margin=2.6cm]{geometry}
\usepackage{amsmath,amsthm,amsfonts,amssymb,amscd,cite,graphicx}

\usepackage{titlesec}
\titleformat{\section}
{\normalfont\fontsize{12}{15}\bfseries}{\thesection}{1em.}{}

\newtheorem{proposition}{Proposition}[section]

\newtheorem{corollary}{Corollary}[section]
\newtheorem{lemma}{Lemma}[section]

\newtheorem{theorem}{Theorem}[section]

\allowdisplaybreaks[4]

\let\oldbibliography\thebibliography
\renewcommand{\thebibliography}[1]{%
  \oldbibliography{#1}%
  \setlength{\itemsep}{-2pt}%
}

\begin{document}

\baselineskip=0.20in

\noindent
{\large \bf On the Hyperbolic Sombor Index and Its Counterpart}\\

\noindent
Abeer M. Albalahi$^1$,  Shibsankar Das$^{2}$, Akbar Ali$^{1}$, Jayjit Barman$^{2}$, Amjad E. Hamza$^1$\\

\noindent
\footnotesize $^1${\it Department of Mathematics,  College of Science, University of Ha\!'il, Ha\!'il, Saudi Arabia}\\
\noindent
 $^2${\it Department of Mathematics, Institute of Science, Banaras Hindu University, Varanasi-221005, Uttar Pradesh, India\/} \\

\setcounter{page}{1} \thispagestyle{empty}

\baselineskip=1.20in

\normalsize

 \begin{abstract}
 \noindent
For a graph $G$ with edge set $E$, let $d(w)$ denote the degree of a vertex $w$ in $G$.
The hyperbolic Sombor index of $G$ is defined by
$$HSO(G)=\sum_{uv\in E}(\min\{d(u),d(v)\})^{-1}\sqrt{(d(u))^2+(d(v))^2}.$$ If $\min\{d(u),d(v)\}$ is replaced with $\max\{d(u),d(v)\}$ in the formula of $HSO(G)$, then the complementary diminished Sombor (CDSO) index is obtained.
For two non-adjacent vertices $v$ and $w$ of $G$, the graph obtained from $G$ by adding the edge $vw$ is denoted by $G+vw$. In this paper, we attempt to correct some inaccuracies in the recent work
[J. Barman, S. Das, Geometric approach to degree-based topological index: hyperbolic Sombor index, {\it MATCH Commun. Math. Comput. Chem.} {\bf95} (2026) 63--94]. We establish a sufficient condition under which $HSO(G+vw) > HSO(G)$ holds,
and also provide a sufficient condition guaranteeing $HSO(G+vw) < HSO(G)$.
In addition, we give a lower bound on $HSO(G)$ in terms of the order and size of $G$. Furthermore, we obtain similar results for the CDSO index.\\[2mm]
 \end{abstract}

\baselineskip=0.30in

\section{Introduction}

Chemical graph theory, a subfield of mathematical chemistry, utilizes the principles and methods of graph theory to model and study molecular structures (see, e.g., \cite{Leite-24,Wagner-18,Trina-book}). Within this theoretical framework, a molecule is represented by a graph in which vertices correspond to atoms and edges to chemical bonds, thereby enabling a rigorous mathematical analysis of molecular properties. Terminology and foundational concepts from both general graph theory and chemical graph theory, as utilized throughout this study, can be found in standard books such as \cite{Bondy-book,Chartrand-16} and \cite{Wagner-18,Trina-book}, respectively.

A key aspect of chemical graph theory is the development and application of molecular descriptors, which are widely used in cheminformatics for tasks such as virtual screening and the prediction of physicochemical properties \cite{Basak-book}. Todeschini and Consonni~\cite{Todeschini-20-book} define a molecular descriptor as ``the final result of a logical and mathematical procedure that transforms chemical information encoded in a symbolic representation of a molecule into a useful number or the result of some standardized experiment.'' When such descriptors are formulated using the structure of molecular graphs, they are commonly referred to as topological indices. For recent advances and applications of topological indices in chemistry, we refer the reader to \cite{Desmecht-JCIM-24,Leite-24}.

Among the various classes of topological indices, degree-based topological indices \cite{KD1,gut04,gutman13degree,borovicanin17zagreb,Ali-IJACM-17,Ali-IJQC-17,nithyaa24} occupy a central role due to their computational simplicity and their good empirical correlation with numerous molecular properties.  One of the most prominent degree-based indices is the Sombor (SO) index, introduced by Gutman in \cite{gutman21geo}. For a simple graph $G$, it is defined as
\[
\mathcal{SO}(G) = \sum_{uv \in E(G)} \sqrt{(d(u))^2 + (d(v))^2},
\]
where $E(G)$ denotes the edge set of $G$, and $d(u)$ represents the degree of vertex $u$. If multiple graphs are under consideration, we write $d_G(u)$ to indicate that the degree is taken in graph $G$. The Sombor index has gained significant attention in recent literature; for instance, see the survey \cite{liu22review} and some of the recent studies \cite{chen24,das24open,new2,new3,new11}.

Gutman also introduced two variants of the SO index in \cite{gutman21geo}, namely the reduced Sombor index and the average Sombor index. Since then, its numerous variants have been proposed, including the elliptic \cite{gutman24elliptic,gutman24eu}, Euler \cite{new7}, Zagreb \cite{Ali-AIMS-25,Ali-AIMS-2024}, hyperbolic \cite{Barman-MATCH-26}, diminished \cite{Movahedi-26}, and augmented \cite{ASO-index} Sombor indices.

The first part (that is, Section 2) of the present paper focuses on the hyperbolic Sombor (HSO) index, which for a graph $G$ is defined as
\[
\mathcal{HSO}(G) = \sum_{uv \in E(G)} \frac{\sqrt{(d(u))^2 + (d(v))^2}}{\min\{d(u), d(v)\}}.
\]
It is noteworthy that the HSO index can be derived from the diminished Sombor (DSO) index by replacing the denominator $d(u) + d(v)$ with $\min\{d(u), d(v)\}$, where the DSO index is defined as
\[
\mathcal{DSO}(G) = \sum_{uv \in E(G)} \frac{\sqrt{(d(u))^2 + (d(v))^2}}{d(u) + d(v)}.
\]
Here, we consider a natural counterpart to the HSO index that arises from the implementation of the framework of complementary topological indices, as described in~\cite{Furtula-MATCH-25}, to the DSO index:
\begin{equation}\label{eq-cDSO}
\frac{1}{\sqrt{2}} \sum_{uv \in E(G)} \frac{\sqrt{(d(u))^2 + (d(v))^2}}{\max\{d(u), d(v)\}}.
\end{equation}
We may drop the factor $1/\sqrt{2}$ from the expression \eqref{eq-cDSO} and call the resulting formula the complementary diminished Sombor (CDSO) index and denote it by $^c\mathcal{DSO}$. Hence, for a graph $G$, we have
\[
^c\mathcal{DSO}(G)=\sum_{uv\in E(G)} \frac{\sqrt{(d(u))^2+(d(v))^2}}{\max\{d(u),d(v)\}}.
\]
The CDSO index is the focus of Section 3 of the present paper.

All graphs considered in this paper are simple and finite. Let $v$ and $w$ be two non-adjacent vertices of a graph $G$. We denote by $G + vw$ the graph obtained by adding the edge $vw$ to $G$.  The vertex set of a graph $G$ is denoted by $V(G)$. If $|V(G)|=n$, we say that $G$ is an $n$-order graph. For a vertex $u \in V(G)$, define  $N(u)=N_G(u) := \{uv \in V(G) : uv \in E(G)\}$.

In this work, we attempt to correct some inaccuracies in the mathematical properties of the HSO index reported in \cite{Barman-MATCH-26}. In order to disprove a statement given in \cite{Barman-MATCH-26}, we establish a sufficient condition under which the inequality $\mathcal{HSO}(G + vw) < \mathcal{HSO}(G)$ holds, and likewise derive a sufficient condition ensuring the reverse inequality $\mathcal{HSO}(G + vw) > \mathcal{HSO}(G)$. Furthermore, we derive a lower bound for $\mathcal{HSO}(G)$ in terms of the order and size of $G$. Furthermore, we obtain results concerning the CDSO index similar to the ones established for the HSO index.

\section{Hyperbolic Sombor Index}

%This section is mainly concerned with Section 3 of the paper \cite{Barman-MATCH-26}. We start with Theorems 1, 2, and 3 of \cite{Barman-MATCH-26}, which contain minor errors in the equality cases of the bounds. The following proposition restates these results, avoiding the mentioned minor errors.
This section is mainly concerned with Section 3 of the paper \cite{Barman-MATCH-26}. We start with Theorems 1, 2, and 3 of \cite{Barman-MATCH-26}, which give bounds for $\mathcal{HSO}$, where it was claimed that equality in any of these bounds holds if and only if the considered graph is a complete graph. Here, we have fine-tuned this claim. We note that most of these equalities hold for a class of regular graphs. The following proposition is a modified version of Theorems 1, 2, and 3 of \cite{Barman-MATCH-26}.

\begin{proposition}\label{prop-1}
Let $G$ be a connected graph of order at least $2$.
\begin{description}
    \item[(i). (Theorem 1 in \cite{Barman-MATCH-26}).] If $G$ has size $m$, then
$\mathcal{HSO}(G)\ge  \sqrt{2}m$,
with equality if and only if $G$ is a regular graph.

    \item[(ii). (Theorem 2 in \cite{Barman-MATCH-26}).] If the minimum and maximum degrees of $G$ are  $\delta$ and  $\Delta$, respectively, then
\begin{equation}\label{eq-1-prop-2}
\frac{1}{\Delta}\mathcal{SO}(G)\le  \mathcal{HSO}(G) \le \frac{1}{\delta}\mathcal{SO}(G).
\end{equation}
The left equality in \eqref{eq-1-prop-2} holds if and only if $G$ is regular. The right equality in \eqref{eq-1-prop-2} holds if and only if every edge of $G$ is incident to at least one vertex of degree $\delta$.

    \item[(iii). (Theorem 3 in \cite{Barman-MATCH-26}).]
If $G$ has maximum degree $\Delta$, then
\begin{equation*}
\mathcal{HSO}(G) \ge \frac{1}{\sqrt{2}\Delta}\mathcal{M}_1(G),
\end{equation*}
with equality if and only if $G$ is regular.

\end{description}

\end{proposition}

The formula concerning $\mathcal{HSO}(P_n)$ in Lemma 1 of \cite{Barman-MATCH-26} does not hold for $n=2$, and hence the constraint ``$n\ge2$'' needs to be replaced with ``$n\ge3$''.

Now, we turn our attention to some mistakes in Theorem 4 of \cite{Barman-MATCH-26}. Although the statement of this theorem is correct, its proof seems to be questionable. Particularly, its proof starts with the following text: ``It is obvious that the value of $\mathcal{HSO}(G)$ increases when we add edges
to the graph $G$. A tree obtains the highest value of $\mathcal{HSO}(G)$ of a connected graph.'' First, we recall that any increasing topological index is maximized by the complete graph $K_n$ (not a tree) over the class of all connected graphs of order $n\ (\ge3)$, where a topological index $TI$ is said to be an increasing topological index if $TI(G+uv)>TI(G)$ for any two non-adjacent vertices $u,v\in V(G)$.

Now, corresponding to the text ``It is obvious that the value of $\mathcal{HSO}(G)$ increases when we add edges
to the graph $G$'', we provide the next result.

\begin{proposition}\label{prop-G+uv<G}
Let $G$ be a graph of minimum degree $\delta~(\ge1)$. Let $u$ and $v$ be non-adjacent vertices of $G$ such that $d(u)=\delta< \min_{u'\in N(u)}\{d(u')\}$ and $d(v)=\delta < \min_{v'\in N(v)}\{d(v')\}$. Then,
$$
\mathcal{HSO}(G+uv)
< \mathcal{HSO}(G).
$$
\end{proposition}

\begin{proof}
Let
$\phi(x,y)= (\min\{x,y\})^{-1}\sqrt{x^2+y^2}$. Since the inequality
\[
\phi(\delta,d_G(w))-\phi(\delta+1,d_G(w))\ge\phi(\delta,\delta+1)-\phi(\delta+1,\delta+1)
\]
holds for any vertex $w\in V(G)$ of degree larger than $\delta$, we have
\begin{align*}
\mathcal{HSO}(G)-\mathcal{HSO}(G+uv)
&= \sum_{u'\in N_G(u)}[\phi(\delta,d_G(u'))-\phi(\delta+1,d_G(u'))]\\[2mm]
&+ \sum_{v'\in N_G(v)}[\phi(\delta,d_G(v'))-\phi(\delta+1,d_G(v'))] \\
& - \phi(\delta+1,\delta+1)\\[2mm]
&\ge 2\delta\,\phi(\delta,\delta+1)-(2\delta+1)\,\phi(\delta+1,\delta+1)>0,
\end{align*}
as desired.
\end{proof}

Along with Proposition \ref{prop-G+uv<G}, we provide the following result.

\begin{proposition}\label{prop-G+uv>G}
Let $u$ and $v$ be non-adjacent vertices of a graph $G$ of minimum degree at least one such that $d(u)\ge \max_{u'\in N(u)}\{d(u')\}$ and $d(v)\ge \max_{v'\in N(v)}\{d(v')\}$. Then
$$
\mathcal{HSO}(G+uv)
> \mathcal{HSO}(G).
$$
\end{proposition}

\begin{proof}
Let
$\phi(x,y)= (\min\{x,y\})^{-1}\sqrt{x^2+y^2}$. The required inequality holds because
\begin{align*}
&\mathcal{HSO}(G)-\mathcal{HSO}(G+uv)\\[2mm]
&= \sum_{u'\in N_G(u)}\bigg[\underbrace{\phi(d_G(u),d_G(u'))-\phi(d_G(u)+1,d_G(u'))}_{negative}\bigg]\\[2mm]
&+ \sum_{v'\in N_G(v)}\bigg[\underbrace{\phi(d_G(v),d_G(v'))-\phi(d_G(v)+1,d_G(v'))}_{negative}\bigg] \\
& - \phi(d_G(u)+1,d_G(v)+1) <0\,.
\end{align*}
\end{proof}

Before proceeding further, we recall the following known result.

\begin{lemma}\label{lem-m->n-1/n}{\rm \cite{Ali-AIMS-25}}
Let $G$ be a connected graph of order $n\ (\ge4)$ and size $m$. Let $\hbar$ be a function defined on the Cartesian square of the set of real numbers greater than or equal to 1 such that $\hbar(x_1,x_2)=\hbar(x_2,x_1)\ge0$ for all $x_1$ and $x_2$ belonging to the domain of $\hbar$ and $\hbar(x_1,x_2)>0$ for $x_1\ne x_2$.
Define the function $\Phi$ on the Cartesian square of the set of positive integers by
\begin{equation*}
\Phi(r_1,r_2):=  \hbar(r_1,r_2) +\frac{2\hbar(1,2)(r_1r_2-r_1-r_2)}{r_1r_2}   +\frac{\hbar(2,2)(2r_1+2r_2-3r_1r_2)}{r_1r_2},
\end{equation*}
with $n-1\ge r_2\ge r_1\ge 1$.
If $\Phi(r_1,r_2)>0$ for $n-1\ge r_2\ge r_1\ge 1$ provided that $(r_1,r_2)\not\in\{(1,1),(1,2),(2,2)\}$, then
\begin{equation}\label{eq-AIMS-paper}
\sum_{uv\in E(G)} \hbar(d_G(u),d_G(v)) \ge 2[\hbar(1,2)-\hbar(2,2)]n + [3\hbar(2,2)-2\hbar(1,2)]m,
\end{equation}
with equality if and only if $G$ is either path graph $P_n$ or cycle graph $C_n$. On the other hand, if $\Phi(r_1,r_2)<0$ for $n-1\ge r_2\ge r_1\ge 1$ such that $(r_1,r_2)\not\in\{(1,1),(1,2),(2,2)\}$, then inequality \eqref{eq-AIMS-paper} is reversed.
\end{lemma}

Here, we remark that the general topological index given on the left-hand side of inequality \eqref{eq-AIMS-paper} is known as a bond incident degree index; for instance, see \cite{Ali-CJC-16}.
Next, using Lemma \ref{lem-m->n-1/n}, we give a lower bound on $\mathcal{HSO}(G)$ in terms of order and size of $G$.

\begin{proposition}\label{prop-path-cycle-minimal}
Let $G$ be a connected graph of order $n\ (\ge4)$ and size $m$. Then,
\begin{equation*}
\mathcal{HSO}(G) \ge 2\left(\sqrt{5}-\sqrt{2}\right)n + \left(3\sqrt{2}-2\sqrt{5}\right)m,
\end{equation*}
with equality if and only if $G$ is either path graph $P_n$ or cycle graph $C_n$.
\end{proposition}

\begin{proof}
We set $\hbar(r,s)=r^{-1}\sqrt{r^2+s^2}$ in the definition of the function $\Phi$ given in Lemma \ref{lem-m->n-1/n} and see that $\Phi(r,s)=(rs)^{-1}\,\Psi(r,s)$, where
\[
\Psi(r,s):=\left(2\sqrt{5}-3\sqrt{2}\right)rs-2\left(\sqrt{5}-\sqrt{2}\right)(r+s)+s\sqrt{r^2+s^2},
\]
provided that $r$ and $s$ are integers satisfying $n-1\ge s\ge r\ge 1$ such that \linebreak $(r,s)\not\in\{(1,1),(1,2),(2,2)\}$.
In what follows, we prove that $\Phi(r,s)>0$. Hence, by Lemma \ref{lem-m->n-1/n}, we will have the required conclusion.

Since $s \ge r \ge 1$, we have
\begin{align}\label{eq-(r,s)}
\Psi(r,s)> F(r,s):&= \left(2\sqrt{5}-3\sqrt{2}\right)rs-2\left(\sqrt{5}-\sqrt{2}\right)(r+s)+s^2.
\end{align}
If $s \ge 8$ is fixed, then $F$ is minimized at $r=1$ and hence inequality \eqref{eq-(r,s)} yields $\Psi(r,s)>F(r,s)\ge F(1,s)>0$.
If $s$ is fixed such that $3\le s \le 8$, then $F$ is minimized at $r=s$ and hence  $\Psi(r,s)>F(r,s)\ge F(s,s)>0$.
\end{proof}

In a graph $G$, a vertex of degree $0$ is known as an isolated vertex of $G$.

\begin{theorem}{\rm \cite{Deng-JCO-15}}\label{Deng-JCO-15-lem}
Let $\hbar$ be the function defined in Lemma \ref{lem-m->n-1/n}. Let $G$ be a graph of order $n\ (\ge3)$ without isolated vertices.
Let
\begin{equation}\label{Deng-1}
S_1:=\{(i, j) : 1 \leq i \leq j \leq n-1\}\setminus\{(1, n-1)\}.
\end{equation}
If
\begin{equation}\label{Deng-2}
 \hbar(i,j)-\frac{n-1}{n}\left(\frac{1}{i}+\frac{1}{j}\right) \hbar(1, n-1)<0
\end{equation}
for every $(i, j) \in S_1$,
then
$$\sum_{uv\in E(G)} \hbar(d_G(u),d_G(v)) \leq(n-1) \hbar(1, n-1),$$ with equality if and only if $G\cong S_n$.

\end{theorem}

\begin{lemma}\label{cor-Sn}
If  $\hbar(i,j):=i^{-1}\sqrt{i^2+j^2}$, then inequality \eqref{Deng-2} holds for every $(i, j) \in S_1$ with $n\ge3$.
\end{lemma}

\begin{proof}
The inequality
\begin{align*}
\hbar(i,j)-\frac{n-1}{n}\left(\frac{1}{i}+\frac{1}{j}\right) \hbar(1, n-1)<0
\end{align*}
is equivalent to
\begin{equation}\label{eq-cor-1}
jn\sqrt{i^2+j^2}-(i+j)(n-1)\sqrt{(n-1)^2+1}<0,
\end{equation}
where $1 \leq i \leq j \leq n-1 $, $(i, j)\ne (1, n-1)$ and $n\ge3$. By keeping in mind \eqref{eq-cor-1}, we define
\begin{equation}\label{eq-H(i,j)}
H(i,j,n):=\frac{n-1}{n}\,\sqrt{(n-1)^2+1}\cdot\frac{i+j}{j\sqrt{i^2+j^2}}.
\end{equation}
It suffices to prove that $H(i,j,n)> 1$ for every $(i, j) \in S_1$. We note that, if $(1, j) \in S_1$ then $1\le j\le n-2$.   Hence,
\[
\frac{i+j}{j\sqrt{j^2+j^2}}
\begin{cases}
=\displaystyle \frac{j+1}{j\sqrt{j^2+1}}\ge \frac{n-1}{(n-2)\sqrt{(n-2)^2+1}}  & \text{if $i=1$}\\[5mm]
\ge\displaystyle\frac{j+2}{j\sqrt{j^2+4}}\ge \frac{n+1}{(n-1)\sqrt{(n-1)^2+4}} & \text{if $i\ge2$.}
\end{cases}
\]
Therefore, from \eqref{eq-H(i,j)}, we have $H(i,j,n)>1$ for $(i,j)\in S_1$ and $n\ge3$.
\end{proof}

As we mentioned before, the statement given in Theorem 4 of \cite{Barman-MATCH-26} is correct, but its proof seems to be questionable.
Next, we give a proof of this theorem.

\begin{proposition}\label{thm-Cn-Sn}
If $G$ is a connected graph of order $n~(\ge3)$, then
\begin{equation}\label{thm-Cn-Sn-eq-1}
 \mathcal{HSO}(C_n) \le   \mathcal{HSO}(G) \le \mathcal{HSO}(S_n).
\end{equation}
The left (right, respectively) equality in \eqref{thm-Cn-Sn-eq-1} holds if and only if $G$ is cycle graph $C_n$ (star graph $S_n$, respectively).
\end{proposition}

\begin{proof}
The right inequality in \eqref{thm-Cn-Sn-eq-1} follows from Theorem \ref{Deng-JCO-15-lem} and Lemma \ref{cor-Sn}.

If $G \cong P_n$, then direct comparison yields $\mathcal{HSO}(C_n) <  \mathcal{HSO}(P_n)$ for $n\ge3$. In what follows, assume that $G \not\cong P_n$. We may assume that $n\ge4$. Let $m$ be the size of $G$. If $m=n-1$, then by Proposition \ref{prop-path-cycle-minimal}, we have $\mathcal{HSO}(C_n)<\mathcal{HSO}(P_n) <   \mathcal{HSO}(G)$. If $m\ge n$, then by Proposition \ref{prop-1}(i), we have $\mathcal{HSO}(C_n) \le   \mathcal{HSO}(G)$ with equality if and only if $G\cong C_n$.
\end{proof}

We also observe that the statement given in Theorem 5 of \cite{Barman-MATCH-26} is correct, but its proof seems to be imperfect; for instance, the proof contains the following text: ``It is obvious that the value of $\mathcal{HSO}(G)$ decreases when we remove edges from the graph $G$'', which contradicts Proposition \ref{prop-G+uv<G} (because $G$ can be obtained from $G+uv$ by removing the edge $uv$). Although this text does not play any role in the proof under consideration, one may find some ways to improve the rest of the mentioned proof.
However, we remark here that the statement given in Theorem 5 of \cite{Barman-MATCH-26} follows from Propositions \ref{prop-path-cycle-minimal} and \ref{thm-Cn-Sn}.

\section{Complementary Diminished Sombor Index}

Here, we provide results concerning the CDSO index similar to the ones given in the preceding section.
We start this section with the observation that $^c\mathcal{DSO}(G)\le \mathcal{HSO}(G),$
with equality if and only if $G$ is regular.

Corresponding to Proposition \ref{prop-1}(i), we provide the following result.

\begin{proposition}\label{prop-size-bound}
If $G$ is a connected graph of size $m\ (\ge1)$, then
$$^c\mathcal{DSO}(G) \le m\sqrt{2},$$
with equality if and only if $G$ is regular. In addition, if the maximum degree of $G$ is $\Delta$, then
$$^c\mathcal{DSO}(G) \ge \frac{m\sqrt{\Delta^2+1}}{\Delta},$$
with equality if and only if $G$ is $S_{m+1}$.
\end{proposition}

\begin{proof}
Let $uv$ be any edge of $G$ such that $d(u)\le d(v)$. Then, the desired conclusions follow from the fact that
\begin{equation}\label{eq-CDSO-2}
\sqrt{1+\left(\frac{1}{\Delta}\right)^2} \le  \sqrt{1+\left(\frac{d(u)}{d(v)}\right)^2} \le \sqrt{2},
\end{equation}
where the right equality in \eqref{eq-CDSO-2} holds if and only if $d(u)=d(v)$, and the left equality in \eqref{eq-CDSO-2} holds if and only if $(d(u),d(v))=(1,\Delta)$.
\end{proof}

The next result follows directly from Proposition \ref{prop-size-bound}.

\begin{corollary}\label{cor-cnctd-graphs}
If $G$ is a connected graph of order $n\ (\ge 3)$, then
\begin{equation}\label{eq-CDSO-cnctd}
\sqrt{(n-1)^2+1}\le\, ^c\mathcal{DSO}(G)   \le \frac{n(n-1)}{\sqrt{2}}
\end{equation}
with left (right, respectively) equality if and only if $G$ is the star graph $S_n$ (complete graph $K_n$, respectively).
\end{corollary}

Similar to parts (ii) and (iii) of Proposition \ref{prop-1}, we have the results presented in the following proposition.

\begin{proposition}\label{prop-1-CDSO}
Let $G$ be a connected graph of order at least $2$.
\begin{description}
    \item[(i).] If the minimum and maximum degrees of $G$ are  $\delta$ and  $\Delta$, respectively, then
\begin{equation}\label{eq-1-prop-2-CDSO-}
\frac{1}{\Delta}\mathcal{SO}(G)\le  \,^c\mathcal{DSO}(G) \le \frac{1}{\delta}\mathcal{SO}(G).
\end{equation}
The right equality in \eqref{eq-1-prop-2-CDSO-} holds if and only if $G$ is regular. The left equality in \eqref{eq-1-prop-2-CDSO-} holds if and only if every edge of $G$ is incident to at least one vertex of degree $\Delta$.
\item[(ii).] If $G$ has maximum degree $\Delta$, then
\begin{equation*}
^c\mathcal{DSO}(G) \ge \frac{1}{\sqrt{2}\Delta}\mathcal{M}_1(G),
\end{equation*}
with equality if and only if $G$ is regular.

\end{description}

\end{proposition}

Next, we provide two results, corresponding to Propositions \ref{prop-G+uv<G} and \ref{prop-G+uv>G}, concerning the difference $
 ^c\mathcal{DSO}(G) -\, ^c\mathcal{DSO}(G+uv)$.

\begin{proposition}\label{prop-G+uv<G-CDSO}
Let $u$ and $v$ be non-adjacent vertices of a graph $G$ such that $d(u)= d(u')=4d(v)=4d(v')\ge12$ for all $u'\in N(u)$ and $v'\in N(v)$. Then,
$$
^c\mathcal{DSO}(G+uv)
<\, ^c\mathcal{DSO}(G).
$$
\end{proposition}

\begin{proof}
Let
$\phi(x,y)= (\max\{x,y\})^{-1}\sqrt{x^2+y^2}$. Then, we have
\begin{align*}
&^c\mathcal{DSO}(G)-\, ^c\mathcal{DSO}(G+uv)\\[2mm]
&= 4d_G(v)[\phi(4d_G(v),4d_G(v))-\phi(4d_G(v)+1,4d_G(v))]\\[2mm]
&+ d_G(v)[\phi(d_G(v),d_G(v))-\phi(d_G(v)+1,d_G(v))] \\[2mm]
& - \phi(4d_G(v)+1,d_G(v)+1)>0,
\end{align*}
because $d_G(v)\ge3$.
\end{proof}

\begin{proposition}\label{prop-G+uv>G-CDSO}
Let $u$ and $v$ be non-adjacent vertices of a graph $G$ of minimum degree at least one such that $d(u)+1\le \min_{u'\in N(u)}\{d(u')\}$ and $d(v)+1\le \min_{v'\in N(v)}\{d(v')\}$. Then
$$
\mathcal{HSO}(G+uv)
> \mathcal{HSO}(G).
$$
\end{proposition}

\begin{proof}
Let
$\phi(x,y)= (\max\{x,y\})^{-1}\sqrt{x^2+y^2}$. Then, the required inequality holds because
\begin{align*}
&\mathcal{HSO}(G)-\mathcal{HSO}(G+uv)\\[2mm]
&= \sum_{u'\in N_G(u)}\bigg[\underbrace{\phi(d_G(u),d_G(u'))-\phi(d_G(u)+1,d_G(u'))}_{negative}\bigg]\\[2mm]
&+ \sum_{v'\in N_G(v)}\bigg[\underbrace{\phi(d_G(v),d_G(v'))-\phi(d_G(v)+1,d_G(v'))}_{negative}\bigg] \\
& - \phi(d_G(u)+1,d_G(v)+1) <0\,.
\end{align*}
\end{proof}

Next, using Lemma \ref{lem-m->n-1/n}, we have the following result.

\begin{proposition}\label{prop-path-cycle-maximal}
Let $G$ be a connected graph of order $n\ (\ge4)$ and size $m$. Then,
\begin{equation*}
^c\mathcal{DSO}(G) \le \left(\sqrt{5}-2\sqrt{2}\right)n + \left(3\sqrt{2}-\sqrt{5}\right)m,
\end{equation*}
with equality if and only if $G$ is either path graph $P_n$ or cycle graph $C_n$.
\end{proposition}

\begin{proof}
By setting $\hbar(r,s)=s^{-1}\sqrt{r^2+s^2}$ in the definition of the function $\Phi$ given in Lemma \ref{lem-m->n-1/n}, we note that
$\Phi(r,s)<0$
 for all integers $r$ and $s$ satisfying $s\ge r\ge 1$ such that $(r,s)\not\in\{(1,1),(1,2),(2,2)\}$.
Therefore, the required result follows from Lemma \ref{lem-m->n-1/n}.
\end{proof}

\begin{proposition}\label{prop:trees-CDSO}
If $T$ is a tree of order $n\ (\ge 4)$, then
\begin{equation}\label{prop:trees-A1}
\sqrt{(n-1)^2+1}\le\, ^c\mathcal{DSO}(T)   \le \sqrt{5}+(n-3)\sqrt{2}
\end{equation}
with left (right, respectively) equality if and only if $T$ is the star graph $S_n$ (path graph $P_n$, respectively).

\end{proposition}

\begin{proof}
The left (right, respectively) inequality in \eqref{prop:trees-A1} follows from Corollary \ref{cor-cnctd-graphs}  (Proposition \ref{prop-path-cycle-maximal}, respectively).
\end{proof}

\section{Concluding Remarks}
In this paper, we discussed some of the mathematical aspects of HSO and CDSO indices. Particularly, graphs maximizing/minimizing the aforementioned indices among the class of all fixed-order connected graphs/trees are known. We recall that trees can be considered connected graphs of cyclomatic number $0$, where the cyclomatic number of a graph is the minimum number of edges whose removal makes the graph acyclic. It is natural to think about the problem of characterizing graphs maximizing/minimizing either the HSO index or the CDSO index over the class of all fixed-order connected graphs with cyclomatic number $\ell\ (\ge1)$. We perform some numerical tests on connected graphs of small orders, which suggest that a graph minimizing (maximizing, respectively) the CDSO index (HSO index, respectively) over the aforementioned class of graphs has a vertex adjacent to all other vertices. Also, we believe that the minimum and maximum degrees of a graph maximizing (minimizing, respectively) the CDSO index (HSO index, respectively) over the class of graphs under consideration, with $\ell\ge2$, belong to the set $\{2,3\}$.

\section*{Acknowledgment}

This research has been funded by the Scientific Research Deanship at the University of Ha\!'il, Saudi Arabia, through project number RG-25\,045.

%%%%%%%%%%%%%%%%%%%%%%%%%%%%%%%%%%%%%%%%%%%%%%%%%%%%%%%%%%%%%%%%%%%%

\end{document}